\documentclass[11pt]{amsart}
\usepackage{amsmath,amsthm,epsfig,amssymb}
\usepackage{amsthm,amsfonts,amssymb,amscd}
\usepackage[utf8x]{inputenc}
\usepackage{textcomp}
\usepackage{tikz}
\usetikzlibrary{positioning, matrix, arrows}
\usepackage[all]{xy}
\usepackage[colorlinks=true,citecolor=blue,urlcolor=blue,linkcolor=blue]{hyperref}
\usepackage{url}
\usepackage{mathtools}
\newtheorem{Theorem}[subsection]{Theorem}
\newtheorem{Proposition}[subsection]{Proposition}

\newtheorem{Lemma}[subsection]{Lemma}

\theoremstyle{definition}
\newtheorem{Remark}[subsection]{Remark}

\newtheorem{proposition-definition}[subsection]{Proposition-Definition}

\DeclarePairedDelimiter\floor{\lfloor}{\rfloor}

\input xy
\xyoption{all}

\begin{document}

\title {Lazarsfeld-Mukai bundles on K3 surfaces associated with a pencil computing  the Clifford index }


\author[S.Pal]{Sarbeswar Pal}

\address{Indian Institute of Science Education and Research, Thiruvananthapuram,
Maruthamala PO, Vithura,
Thiruvananthapuram - 695551, Kerala, India.}

\email{sarbeswar11@gmail.com,  spal@iisertvm.ac.in}

\date{}

\begin{abstract}
Let $X$ be a  smooth  projective K3 surface over the complex numbers and let $C$ be an ample curve on $X$. In this paper we will study the semistability of the Lazarsfeld-Mukai bundle $E_{C, A}$ associated to a line bundle  $A$ on $C$ such that $|A|$
is a pencil on $C$ and computes the Clifford index of $C$. We give a necessary and sufficient condition for $E_{C, A}$ to be semistable. 
\end{abstract}

\maketitle

\section{Introduction}

Let $X$ be a smooth projective K3 surface over the complex numbers and $C$ be a smooth projective ample curve in $X$. Given a globally generated line bundle $A$ on $C$,
the Lazarsfeld-Mukai bundle $E_{C, A}$ is defined as the dual of the kernel of the evaluation map
\[
H^0(A) \otimes \mathcal{O}_X \to \iota_*(A),
\]
 where $\iota : C \to X$ is the inclusion.  To study the behavior  of certain invariants of a curve $C$ along a linear system  (for example Clifford index, gonality, gonality sequence etc)  it has been essential to investigate the properties of the Lazarsfeld-Mukai bundles. Several authors have studied the Lazarsfeld-Mukai bundle in several contexts for example see \cite{CP}, \cite{GL}, \cite{S}, \cite{LC} and the references therein.  \\
 The study of the stability (Gieseker or slope) of vector bundles with respect to a given ample line bundle on algebraic varieties is a very active topic in algebraic geometry. 
The purpose of this paper is to study the stability  of the Lazarsfeld-Mukai  bundle $E_{C, A}$ associated to a curve $C$ on  a K3 surfaces $X$ and a globally generated line bundle
$A$ on $C$. In \cite{LC} Margherita Lelli-Chiesa  proved that if $C$ is a $\floor{\frac{g+3}{2}}$-gonal curve of genus 
$g$ and Clifford dimension one and degree of $A = d$ satisfies $\rho(g, 1, d) = 2d -g -2 > 0$, then $E_{C, A}$ is stable with respect to $\mathcal{O}_X(C)$.  In \cite{KW}, Watanable has shown 
that if $E_{C, A}$ is not slope semistable with respect to $\mathcal{O}_X(C)$, the maximal destabilizing subsheaf of it contains an initialized and ACM line bundle with respect to $\mathcal{O}_X(C)$, which gives a sufficient condition for $E_{C, A}$ to be $\mathcal{O}_X(C)-$slope semistable and gave some examples. 

In this article, we will show that if the Clifford index of $C$ is strictly smaller than $\floor{\frac{g-1}{2}} -1$ and $A$ is a pencil on $C$ computing the Clifford index of $C$, then $E_{C, A}$
is never $\mathcal{O}_X(C)-$slope stable. In fact we shall show that if $E_{C, A}$ is semistable then it is free, that is, $E_{C, A}$ is a direct sum of line bundles of the same slope. Hence $C$ must have a decomposition of the form $C \sim 2D$, where $D$ is an effective divisor.  More precisely we will prove the following theorem:
\begin{Theorem}\label{MT}
Let $C$ be a smooth projective curve on a smooth projective K3 surface $X$ of Clifford index $< [\frac{g-1}{2}] -1$. Let $A$ be a globally generated pencil on $C$ computing the 
Clifford index of $C$. If $E_{C, A}$ is semistable then it splits as a direct sum of line bundles of the same slope with respect to $\mathcal{O}_X(C)$ and $C$ decomposes as $C \sim 2D$
for some effective divisor $D$ and the Clifford index of $C$, is $[\frac{g-1}{2}]-2$. In particular $E_{C, A}$ can never be stable. Furthermore, if the Clifford index of $C$,  is 
strictly smaller that $[\frac{g-1}{2}] -2$, then $E_{C, A}$ is never semi-stable. 
Conversely, if the Clifford index of $C$, is equals to $[\frac{g-1}{2}]-2$ and $C$ can be decomposed as $C \sim 2D$ but $C \nsim 2D^\prime + kE + k\Delta$, where $k$ is a positive integer,  $E$ is a smooth elliptic curve and $\Delta$ is a $(-2)$ curve, then for a pencil $A$ computing Clifford index of $C E_{C, A}$ is semistable.
\end{Theorem} 

{\bf Notations and Conventions}
We work over the complex number field $\mathbb{C}$. Surfaces and curves are smooth and projective. For a  curve $C$, we denote by $K_C$ the canonical line bundle of $C$. For a line bundle $L$ on a smooth 
projective variety $X$, we denote  by $|L|$  the linear system defined by $L$, i.e., $|L|= \mathbb{P}({H^0(L)}^*)$. \\
For a line bundle $A$ on a  curve $C$, the Clifford index of $A$ is defined as follows;
\begin{center}
$\text{Cliff}(A) := \text{degree}(A) - 2  \text{dim}(|A|)$.
\end{center}
The Clifford index of a  curve $C$ is defined as follows;
\begin{center}
$\text{Cliff}(C) :=\text{min}\{ \text{Cliff}(A) | h^0(A) \ge 2, h^1(A) \ge 2\}$.
\end{center}
Clifford's theorem states that $\text{Cliff}(C) \ge 0$ with equality if and only if $C$ is
hyperelliptic, and similarly $\text{Cliff}(C)= 1$ if and only if $C$ is trigonal or a smooth
plane quintic. At the other extreme, if $C$ is a general curve of genus $g$ then
$\text{Cliff}(C)= [(g-1)/2]$, and in any event $\text{Cliff}( C) \le [(g -1)/2]$. We say that a line
bundle $A$ on $C$ contributes to the Clifford index of $C$ if $A$ satisfies the
inequalities in the definition of $\text{Cliff}(C)$;  it computes the Clifford index of $C$ if in
addition $\text{Cliff}(C) = \text{Cliff}(A)$.
\section{Linear system on K3 surfaces}
In this section we recall some classical results about line bundles and divisors on K3 surfaces.
\begin{Proposition}\label{P1}
Let $D$ be a non-zero effective divisor with $D^2 \ge 0$ on a K3 surface $X$. Then one can write $D \sim D^{'} + \Delta$, where $\Delta$ is the fixed component of $D$ and $D^{'}$ is base point free. Let $\Delta_1, \Delta_2, ... \Delta_n$ be the connected reduced components of $\Delta$.  Then one of the following holds:\\
(i) there exists an elliptic 
curve $E$, such that $D^{'} \sim kE $, for some integer $k \ge 2$  and there exist one and only one connected reduced component $\Delta_1$ of $\Delta$ such that $E\cdot \Delta_1= 1$ and  $\Delta_i\cdot E= 0,$ for $i \ne 1$.\\
(ii) $D^{'}$ is an irreducible and $D^{'}. \Delta_i = 0 \text{ or } 1, i= 1, 2, ..., n$. 

\end{Proposition}
\begin{proof}
See \cite[2.7]{SD}.
\end{proof}
\begin{Remark}\label{RZ}
(a) Note that any connected reduced component of $\Delta$ has self intersection number   $\le -2$.   Thus if $\Delta_1$ is a connected reduced component of $\Delta$, and the second case occurs then
$D. \Delta_1 < 0$. Thus if $D$ is  nef, then we get a contradiction. In other words, if $D$ is nef then case (ii) does not occur. \\
Similarly, in case (i), one can see that  if $D$ is nef, then $D \sim kE + \Delta$, where $k \ge 2$ and $\Delta$ is an irreducible $(-2)$ curve.\\
(b) Further more if $D^2 = 0$, then $\Delta = 0$, hence $D \sim kE$ for some non-negative integer $k$. 
\end{Remark}
\begin{Proposition}\label{P2}
Let $L$ be a line bundle on a K3 surface $X$ such
that $|L| \ne  \varnothing$ and such that $|L|$ has no fixed components. Then either\\
(i) $L^2 > 0$, and the generic member of $|L|$ is an irreducible curve of
arithmetic genus $\frac{1}{2}(L.L) + 1$. In this case $h^1(L) = 0$, or\\
(ii) $L^2=0$, then $L \cong (\mathcal{O}_X(E))^{\otimes k}$, where $k$ is an integer $\ge 1$ and $E$ an
irreducible curve of arithmetic genus $1$. In this case $h^1(L) = k-1$  and every
member of  $|L|$ can be written as a sum $E_1+ E_2 + ... +E_k$, where $E_i \in |E|$  for $i= 1, 2, ..., k$.
\end{Proposition}
\begin{proof}
See \cite[Proposition 2.6]{SD}.
\end{proof}
\begin{Proposition}\label{P3}
Let $L$ be a line bundle on a K3 surface $X$ such that $|L| \ne \varnothing$. Then $|L|$ has no base points outside its fixed components.
\end{Proposition}
\begin{proof}
See \cite[ Corollary 3.2]{SD}.
\end{proof}
\begin{Theorem}\label{T1}
Let $|L|$ be a complete linear system on a K3 surface $X$,
without fixed components, and such that $L^2 \ge  4$. Then $L$ is hyperelliptic only in
the following cases:\\
(i) There exists an irreducible curve $E$ such that $p_a (E) =1$  and  $E. L = 1  \text{ or } 2$.\\
(ii) There exists an irreducible curve $B$ such that $p_a (B) = 2$  and $L \cong \mathcal{O}_X(2B)$.
\end{Theorem}
\begin{proof}
See \cite[Theorem 5.2]{SD}.
\end{proof}
\section{Structure of Lazarsfeld-Mukai bundles}
In this section we recall the basic properties of the bundle $E_{C,  A}$ of Lazarsfeld  \cite{L} , associated to an irreducible smooth curve $C$ in $X$ and a globally generated  line bundle $A$. 

Let $X$ be  $K3$ surface.  Let $C$ be an irreducible smooth curve in $X$ and let $A$ be a globally generated line bundle on $C$. Viewing $A$ as a sheaf on $X$, consider the evaluation map:
\[
H^0(C, A) \otimes \mathcal{O}_X \to A.
\]
Let $F_{C, A}$ be its kernel and $E_{C, A} := F_{C, A}^*$.  Then $F_{C, A}$ fits in the following exact sequence on $X$.
\begin{equation}\label{e1}
0 \to F_{C, A} \to H^0(C, A) \otimes \mathcal{O}_X \to A \to 0.
\end{equation}
It is easy to check that $F_{C, A}$ is locally free. Dualizing the above exact sequence one gets 
\begin{equation}\label{e2}
0 \to {H^0(C, A)}^* \otimes \mathcal{O}_X \to E_{C, A} \to \mathcal{O}_C(C) \otimes A^* \to 0.
\end{equation}
Then it is easy to check that the following properties hold:
\begin{Lemma}\label{L1}
1. Rank of $E_{C, A} = h^0(C, A)$.\\
2. $\text{det}(E_{C, A}) = \mathcal{O}_X(C)$.\\
3. $c_2(E_{C, A}) = \text{deg}(A)$.\\
4. $h^0(X, {E_{C, A}}^*)= h^1(X, {E_{C, A}}^*) = 0$.\\
5. $E_{C, A}$ is generated by its global sections off a finite set.\\
6. If $\rho(g, d, r) = g -(r+1)(g-d+r) < 0$, then $E_{C, A}$ is non-simple.
\end{Lemma}

Furthermore if $E_{C, A}$ is of rank $2$, that is,  $|A|$ is a pencil then $E_{C, A}$ has the following characterization.
\begin{Lemma}\label{L2}
Let $\mathcal{F}$ be a non-simple vector bundle of rank $2$ on $X$. There exists line bundles $M, N$ on $X$ and a zero-dimensional subscheme $Z \subset X$
such that $\mathcal{F}$ fits in an exact sequence
\begin{equation}\label{e3}
0 \to M \to \mathcal{F} \to N \otimes \mathcal{I}_Z \to 0
\end{equation}
and either\\
(a) $M \ge N$ or \\
(b) $Z$ is empty and the sequence splits, $\mathcal{F} \approx M \oplus N$.
\end{Lemma}
\begin{proof}
See \cite[Lemma 4.4]{DM}.
\end{proof}
\section{The main theorem}
In this section we will prove the main theorem. \\
\begin{Lemma}\label{L3}
Let $C$ be an ample curve of genus $g$ on $X$ such that $C$ can be decomposed as $C \sim C_1 + C_2$ such that $C_1.C= C_2.C$ with $C_1.C_2 \le  [\frac{g-1}{2}]$.  If $C$ has another decomposition 
$C \sim D_1 + D_2$  with $D_1.D_2 \le [\frac{g-1}{2}]$ such that $D_1.C > D_2.C$, then $D_1.D_2 \le C_1.C_2$, and the equality holds if and only if $D_1 \sim D_2 + kE + k\Delta$ where $k$ is an  positive integer,$E$ 
is an elliptic curve and $\Delta$ is either zero or a connected reduced $(-2)$ curve. 
\end{Lemma}
\begin{proof}
Note that since $C_1.C_2 \le [\frac{g-1}{2}],  {(C_1-C_2)}^2 = C_1^2 +C_2^2 -2C_1.C_2= C^2 -4C_1.C_2 \ge 0$. Thus $\mathcal{O}_X(C_1 - C_2)$ has a section up to exchanging $C_1$ and $C_2$. Then we have two possibilities. (i) $C_1 -C_2 \nsim 0$,
then $C_1 - C_2 \sim D$ for some effective divisor $D$.\\
Or (ii), $C_1 -C_2 \sim 0$.\\
If $D$ is non-zero effective then since $C$ is ample, $C. D > 0$, a contradiction. Thus $C_1 \sim C_2$. 
Since $D_1.D_2 \le   [\frac{g-1}{2}]$ and $D_1 \nsim D_2$,  as before $D_1 -D_2 \sim D$ for some  non-zero effective divisor $D$. Thus we have
\[
2D_1 \sim C + D \text{ and } 2D_2 \sim C - D,
\]
which gives that $D_1^2 + D_2^2 = \frac{C^2}{2} + \frac{D^2}{2}$.
Thus we have 
\begin{equation}
\begin{split}
C^2  & = D_1^2 + D_2^2 + 2 D_1.D_2 \\
        & = \frac{C^2}{2} + \frac{D^2}{2} + 2D_1.D_2   \text{ which implies } \\
        & \frac{C^2}{2} - \frac{D^2}{2} = 2D_1.D2.
\end{split}
\end{equation}
On the other hand, $\frac{C^2}{2} = 2 C_1.C_2$. Thus we have 
\begin{equation}\label{E_1}
2C_1.C_2= 2D_1.D_2 +\frac{D^2}{2}.
\end{equation}  
Since $D^2 \ge 0, D_1.D_2 \le C_1.C_2$ and equality holds if and only if $D^2 = 0$. By Remark \ref{RZ}  $D \sim kE + k \Delta$ ($\Delta$ could be 0)  where $k$ is a positive integer and $E$ is a smooth elliptic curve and $\Delta$ is a $(-2)$ curve, which concludes the Lemma. 
\end{proof}
\begin{Lemma}\label{L4}
Let $C$ be an ample curve on a K3 surface $X$ with Clifford index $ < [\frac{g-1}{2}]$.  Let $A$ be a globally generated pencil on $C$ computing the Clifford index. Then $E_{C, A}$ fits
in an exact sequence of the form
\[
0 \to M \to E_{C, A} \to N \to 0
\]
where $M, N$ are line bundles with $h^0(M), h^0(N) \ge 2$ and $N$ is base point free.
\end{Lemma}
\begin{proof}
   Let $A$ be a line bundle on $C$ of degree $d$ such that 
$h^0(A) = 2$ and $A$ computes the Clifford index of $C$. Note that $ c= d-2$.  By Riemann-Roch, we have $h^0(K_C \otimes A^*) = g-c-1$. Thus from the exact sequence
\eqref{e2}, we have
\[
 h^0(E_{C, A}) = 2 + h^0(K_C \otimes A^*) = g +1 -c.
 \]
 On the other hand, by Lemma \ref{L2}, there are line bundles $M, N$ satisfying the hypothesis of the Lemma and a zero-dimensional subscheme
  $Z \subset X$ such that $E_{C, A}$ fits in the following 
 exact sequence:
 \[
 0 \to M \to E_{C, A} \to N \otimes \mathcal{I}_Z \to 0.
 \]
 Note that since $E_{C, A}$ is globally generated, $N$ is also globally generated, hence $h^0(N) \ge 2$. If $Z$ is empty, then $M$ is also globally generated and if $Z$ is non-empty, then $M \ge N$. Thus in any case $h^0(M), h^0(N) \ge 2$.
 Let us assume that $Z$ is non-empty.  
  It is easy to see that $M_{\mid_C}$ computes the Clifford index of $C$. Thus we have 
 \[
 c= M.C +2 - 2h^0(M_{\mid_C})  \text{ which gives } h^0(M_{\mid_C})= \frac{M.C}{2} +1 - \frac{c}{2}.
 \]
On the other hand, by Riemann-Roch, we have
\begin{align*}
h^0(N_{\mid_C})= N.C  +1 - g + h^1(N_{\mid_C})= N.C +1 -g + h^0(M_{\mid_C}) \\
                                                                               = N.C  +1 -g +  \frac{M.C}{2} +1 - \frac{c}{2}.
\end{align*}
 Since $E_{C, A}$ is globally generated off a finite set and by Proposition \ref{P3}, $N$ cannot have base points outside a fixed component,
$N$ is base point free. Thus if $Z$ is nonempty, then $h^0(N \otimes \mathcal{I}_Z) < h^0(N)$.  Thus we have,
\begin{equation}
\begin{split}
g + 1 - c  & = h^0(E_{C, A})\\
               & \le h^0(M) + h^0(N \otimes \mathcal{I}_Z)\\
               & < h^0(M) + h^0(N) \\
               & \le h^0(M_{\mid_C}) + h^0(N_{\mid_C})\\
               & = \frac{M.C}{2} +1 - \frac{c}{2} + N.C  +1 -g +  \frac{M.C}{2} +1 - \frac{c}{2}\\
               & =  M.C + N.C + 3 -g - c\\
               & = (M + N).C + 3 - g - c\\
               & =C^2  + 3 - g -c\\
               & =2g - 2 + 3 - g - c\\
               & =g + 1 - c,
\end{split}
\end{equation}
a contradiction.
Thus $E_{C, A}$ fits in the following exact sequence
\begin{equation}\label{e4}
0 \to M \to E_{C, A} \to N \to 0.
\end{equation}
\end{proof}
\begin{Remark}\label{RY}
Note that if the sequence, in the above Lemma does not split, then $M \ge N$. Thus $(M-N).C \ge 0$ for any irreducible curve $C$. In other words, $M \otimes N^*$ is nef. 
\end{Remark}
{\bf Proof of the  Theorem \ref{MT}}:\\
By Lemma \ref{L4}, $E_{C, A}$ fits in an exact sequence of the form \ref{e4}.
 If $M \sim N$ then, $h^1(M \otimes N^*) = h^1(\mathcal{O}_X) = 0$. Hence the sequence splits and we are done. Let us assume $M \nsim N$.
Note that $M.N = c_2(E_{C, A}) = d$. \\
Now 
\begin{equation}
\begin{split}
{c_1(M \otimes N^*)}^2 = M^2 + N^2 - 2 M.N  & = M^2 +N^2 + 2M.N -4M.N\\
                                                                         & = {c_1(M \otimes N)}^2 - 4d\\
                                                                         & = C^2 -4d\\
                                                                         & = 2g -2 -4d  \ge 0,   \text{ since } c = d -2 < [\frac{g-1}{2}]-1.
\end{split}
\end{equation}    
Since $M \ge N$, the Euler characteristic computation says that $M \otimes N^*$  has a section. In other words, $|M \otimes N^*|$ contains an effective divisor.\\  
{\bf Claim: $h^1(M \otimes N^*) = 0$}:\\
{\bf Proof of the claim}:\\
By Remark \ref{RY}, $M \otimes N^*$ is nef.  Thus if  $M \otimes N^*$ is not base point free, then by Remark \ref{RZ}, there exists a smooth elliptic curve $E$ and a rational curve $\Gamma$ such that
$M \otimes N^* \cong \mathcal{O}_X(kE + \Gamma)$, where $k$ is an integer $\ge 2$ and $E . \Gamma =1$. \\ 
But $h^0(\mathcal{O}_X(kE + \Gamma))= k+1$. Thus the Euler characteristic computation says that $h^1(\mathcal{O}_X(kE + \Gamma)) = 0$.\\
Let us assume  $M \otimes N^*$ is base point free.\\
If ${c_1(M \otimes N^*)}^2 > 0$, then by Proposition \ref{P2}, $h^1(M \otimes N^*) = 0$ and we are done in this case.
If ${c_1(M \otimes N^*)}^2 = 0$, then by Proposition \ref{P2}, $M \otimes N^* \cong \mathcal{O}_X(kE)$, where $k$ is a positive integer and $h^1(M \otimes N^*) = k-1$.\\
Note that $M^{\otimes 2} = \mathcal{O}_X(C + kE)$. Thus we have $2M.C = C^2 + k C.E > 2g - 2.$ On the other hand, since $E_{C, A}$ is semistable with respect to
 $\mathcal{O}_X(C), M.C \le g-1$ which is  a contradiction. Thus we have $h^1(M \otimes N^*) = 0$, in other words, the sequence \eqref{e4} splits 
and $E_{C, A} \cong M \oplus N$.  \\

Since $E_{C, A}$ is semistable, $M.\mathcal{O}_X(C) = N. \mathcal{O}_X(C)$.  Thus $C$ has a decomposition of the form $C \sim C_1 + C_2$, where $\mathcal{O}_X(C_1) = M$ and
$\mathcal{O}_X(C_2) = N$ and $C.C_1 = C.C_2$. From the proof of Lemma \ref{L3}, we have $C_1 \sim C_2$. Thus we have $C \sim 2D$ for some effective divisor $D$ and 
\[
\text{Cliff}(C) = d-2 = M.N-2 = D^2-2 = \frac{2g-2}{4}-2= \frac{g-1}{2}-2 .
\]
That proves the first part of the Theorem \ref{MT}.\\
Conversely, let $ C \sim 2D$ but $C \nsim 2D^\prime +kE$ for any positive integer $k$ and for any elliptic curve $E$  and $ \text{Cliff}(C) = [\frac{g-1}{2}]-2$. Let $A$ be a globally 
generated pencil computing the Clifford index of $C$. Let us assume that $E_{C, A}$ is not semistable. 
If the subbundle $M$ in \ref{e4} is not destabilizing, that is $M.\mathcal{O}_X(C) \le g-1$, then from the first part of the proof one can see that the sequence \ref{e4} splits.
Thus $E_{C, A} = M \oplus N$ and $N$ destabilizes $E_{C, A}$.  Therefore in any case either $M.\mathcal{O}_X(C)$ or $N. \mathcal{O}_X(C)$ is bigger than or equals to $g$.
   Write $M=\mathcal{O}_X(D_1)$  and $N = \mathcal{O}_X(D_2)$, where $D_1, D_2$ are effective divisors.\\
   Without loss of generality we assume, that 
  $C. D_1 > C.D_2$.   . Thus we have a decomposition of $C$ of the form $C \sim D_1 + D_2$ with $C.D_1 > C.D_2$. By Lemma \ref{L3}, we have $D_1.D_2 \le D^2= \frac{g-1}{2}$. 
  If $D_1.D_2 < \frac{g-1}{2}$, then $\text{Cliff}(C) < \frac{g-1}{2}-2$, a contradiction. Thus $D_1.D_2 = D^2$. Again by Lemma \ref{L3}, this can happen if and only if
   $C \sim 2D^\prime + kE + k\Delta$ for some positive integer $k$, where  $E$ is  an elliptic curve and $\Delta$ is a $(-2)$ curve.  But by hypothesis, that is not possible. Hence we get a contradiction. Therefore $E_{C, A}$ is semistable.
   \section{Examples}
   
   Let $X$ be the K3 surface given by a smooth quartic hypersurface in $\mathbb{P}^3$. Let $C$ be a quadric hypersurface section. In other words, $C$ is a complete intersection of two hypersufaces of degree $4$ and $2$ respectively.  Clearly $C$ is an ample curve in $X$. Then we have the following facts \cite[p.199, F-2]{ACGH}:\\
$\bullet $ $  W^1_3(C) = \varnothing$\\
$\bullet$ $W^1_4(C) \ne \varnothing$ \\
$\bullet$ $W^3_8(C) \ne \varnothing$\\
$\bullet$ $W^3_8(C)-W_2(C)) \subset W^1_6(C) $\\
$\bullet$ $W^2_7(C) = W^3_8(C) - W_1(C)$.\\
Thus the Clifford index of $C$ is $2$ and computed by a line bundle $A$ of degree $4$. Also note that the genus of the curve $C$ is $9$. Thus the Clifford index of $C$ satisfies the
hypothesis of the Theorem \ref{MT}.  It is easy to check that $C$ has a decomposition as $C \sim 2D$ where $D$ is hyperplane section of $X$.  Thus $C$ and $A$ satisfies the hypothesis of the Theorem \ref{MT}. Hence $E_{C, A}$ is semistable.

\end{document}